\documentclass[leqno,12pt]{amsart}
\usepackage{amsfonts}
\usepackage{amsmath,amssymb,amsthm}

\everymath{\displaystyle}%% Added by TTTTurgay

\newcommand{\E}{\mathbb E}
\newcommand{\tr}{\mathrm{tr}}
\newcommand{\ds}{\displaystyle}

\newtheorem{thm}{Theorem}[section]
\newtheorem{cor}[thm]{Corollary}
\newtheorem{lem}[thm]{Lemma}
\newtheorem{prop}[thm]{Proposition}
\theoremstyle{definition}

\theoremstyle{remark}
\newtheorem{rem}[thm]{Remark}
\newtheorem{Example}[thm]{Example}

\begin{document}

\title[Quasi-minimal Lorentz Surfaces with Pointwise 1-type Gauss Map]
{Quasi-minimal Lorentz Surfaces with Pointwise 1-type Gauss Map in Pseudo-Euclidean 4-Space}

\author{Velichka Milousheva, Nurettin Cenk Turgay}

\address{Institute of Mathematics and Informatics, Bulgarian Academy of Sciences,
Acad. G. Bonchev Str. bl. 8, 1113, Sofia, Bulgaria;   "L.
Karavelov" Civil Engineering Higher School, 175 Suhodolska Str., 1373 Sofia, Bulgaria}
\email{vmil@math.bas.bg}

\address{Istanbul Technical University, Faculty of Science and Letters, Department of Mathematics,
34469 Maslak, Istanbul, Turkey}
\email{turgayn@itu.edu.tr}

\subjclass[2010]{Primary 53B30, Secondary 53A35, 53B25}
\keywords{Pseudo-Euclidean space, Lorentz surface, quasi-minimal surface, finite type Gauss map, parallel mean curvature vector field}

\begin{abstract}
A Lorentz surface in the four-dimensional pseudo-Euclidean space  with neutral metric is called quasi-minimal if
its mean curvature vector is lightlike at each point. In the present paper we obtain the complete classification of quasi-minimal Lorentz surfaces  with pointwise 1-type Gauss map.
\end{abstract}

\maketitle

\section{Introduction}

In the present paper we study Lorentz surfaces in pseudo-Euclidean space $\E^4_2$.
A surface is called \textit{minimal} if its mean curvature vector vanishes identically.
Minimal surfaces are important in differential geometry as well as in physics.
Minimal Lorentz surfaces in $\mathbb{C}^2_1$ have been classified recently by  B.-Y. Chen \cite{Chen-TaJM}.
Several classification results for minimal Lorentz surfaces in indefinite space forms are obtained in \cite{Chen3}. In particular, a complete classification of all minimal Lorentz surfaces in a pseudo-Euclidean space $\E^m_s$ with arbitrary dimension $m$ and arbitrary index $s$ is given.

A natural extension of minimal surfaces are quasi-minimal surfaces. A surface in a pseudo-Riemannian manifold is called \textit{quasi-minimal} (also pseudo-minimal or marginally trapped) if its mean curvature vector is lightlike at each point of the surface \cite{Rosca}.
Quasi-minimal surfaces in pseudo-Euclidean space have been very actively studied in the last few years.
In \cite{Chen-JMAA}  B.-Y. Chen classified
quasi-minimal Lorentz  flat surfaces in $\E^4_2$ and gave a complete classification of biharmonic Lorentz surfaces
in  $\E^4_2$ with lightlike mean curvature vector. Several other
families of quasi-minimal surfaces  have also been classified. For
example, quasi-minimal surfaces with constant Gauss curvature in
$\E^4_2$ were classified in \cite{Chen-HMJ, Chen-Yang}. Quasi-minimal
Lagrangian surfaces and quasi-minimal slant surfaces in complex
space forms were classified, respectively, in \cite{Chen-Dillen} and
\cite{Chen-Mihai}. The classification of quasi-minimal surfaces with parallel mean
curvature vector in  $\E^4_2$ is obtained
in \cite{Chen-Garay}.
In \cite{GM5} the classification of quasi-minimal rotational
surfaces of elliptic, hyperbolic or parabolic type is given.
For an up-to-date survey on  quasi-minimal
surfaces, see also \cite{Chen-TJM}.

Another basic class of surfaces in Riemannian and pseudo-Riemannian geometry are the surfaces with parallel mean curvature vector field, since they are critical points of some natural functionals and  play important role in differential geometry,  the theory of harmonic maps, as well as in physics.
Surfaces with parallel mean curvature vector field in Riemannian space forms were classified in the early 1970s by Chen \cite{Chen1} and Yau  \cite{Yau}. Recently, spacelike surfaces with parallel mean
curvature vector field  in arbitrary indefinite space forms were classified in \cite{Chen1-2} and  \cite{Chen1-3}.
A complete classification of Lorentz surfaces with parallel mean curvature vector field in arbitrary pseudo-Euclidean space $\E^m_s$ is given in \cite{Chen-KJM,Fu-Hou,HouYang2010}. A survey on classical and recent results concerning
submanifolds with parallel mean curvature vector in Riemannian manifolds
as well as in pseudo-Riemannian manifolds is presented in \cite{Chen-survey}.

The study of submanifolds of Euclidean or pseudo-Euclidean
space via the notion of finite type immersions began in the late
1970's with the papers \cite{Ch1,Ch2} of B.-Y. Chen.  An isometric immersion $x:M$
$\rightarrow $ $\E^{m}$ of a submanifold $M$ in Euclidean
$m$-space $\E^{m}$  (or  pseudo-Euclidean space $\E^m_s$)  is said
to be of \emph{finite type} \cite{Ch1}, if $x$ identified with the
position vector field of $M$ in $\E^{m}$ (or $\E^m_s$) can be
expressed as a finite sum of eigenvectors of the Laplacian $\Delta
$ of $M$, i.e.
\begin{equation*}
x=x_{0}+\sum_{i=1}^{k}x_{i},
\end{equation*}
where $x_{0}$ is a constant map, $x_{1},x_{2},...,x_{k}$ are
non-constant maps such that $\Delta x_i=\lambda _{i}x_{i},$
$\lambda _{i}\in \mathbb{R}$, $1\leq i\leq k.$ More precisely, if $\lambda
_{1},\lambda _{2},...,\lambda _{k}$ are different, then $M$ is
said to be of \emph{$k$-type}. Many results on finite type
immersions have been collected in the survey paper \cite{Ch3}. The newest results on
submanifolds of finite type are collected in \cite{Chen-book}.

The notion of finite type immersion is naturally extended to the
Gauss map $G$ on $M$ by  B.-Y. Chen and  P. Piccinni  in \cite{CP}, where they introduced
the problem ``\textit{To what extent does the type of the Gauss map of a submanifold of $\mathbb E^m$ determine the submanifold?''}. A submanifold $M$ of an  Euclidean (or pseudo-Euclidean)
space is said to have  \emph{1-type Gauss map} $G$, if $G$
satisfies $\Delta G=a (G+C)$ for some $a \in \mathbb{R}$ and some
constant vector $C$.

A submanifold $M$  is
said to have \emph{pointwise 1-type Gauss map} if its Gauss map
$G$ satisfies
\begin{equation} \label{PW1typeEq1}
\Delta G=\phi (G+C)
\end{equation}
for some non-zero smooth function $\phi $ on $M$ and some
constant vector $C$  \cite{CCK}. A pointwise 1-type Gauss map is called
\emph{proper} if the function $\phi $  is non-constant. A
submanifold with pointwise 1-type Gauss map is said to be of \emph{first kind} if the vector $C$ is zero. Otherwise, it is said
to be of \emph{second kind}.

Classification results on surfaces with pointwise 1-type Gauss map
in Minkowski space have been obtained in the last few years. For
example, in \cite{KY3} Y.  Kim and  D. Yoon studied ruled surfaces with
1-type Gauss map in Minkowski space $\E^m_1$ and gave a complete
classification of null scrolls with 1-type Gauss map. The
classification of ruled surfaces with pointwise 1-type Gauss map
of first kind in Minkowski space $\E^3_1$ is given in
\cite{KY2}. Ruled surfaces with pointwise 1-type Gauss map  of
second kind in Minkowski 3-space were classified in \cite{CKY}.

The complete classification of flat rotation surfaces with  pointwise
1-type Gauss map  in the 4-dimensional pseudo-Euclidean space
$\E^4_2$ is given in \cite{KY2-a}. A classification of
flat Moore type rotational surfaces in terms of the type of their Gauss map is obtained in \cite{Aksoyak15}.
Recently, Arslan and the first author have obtained
 a classification of meridian surfaces with pointwise 1-type Gauss map  \cite{Arslanetall2014}.
The classification of marginally trapped surfaces with pointwise 1-type Gauss
map in Minkowski 4-space is  given in \cite{Mil} and \cite{Turg}.

In the present paper we study quasi-minimal Lorentz surfaces in $\E^4_2$ with pointwise 1-type Gauss
map. First we describe the quasi-minimal surfaces with harmonic Gauss map proving that each such surface is a flat surface with parallel mean curvature vector field.
Next we give explicitly all flat quasi-minimal surfaces with pointwise 1-type Gauss map (Theorem \ref{PropFlatQuasi}).
Further, we obtain that a non-flat quasi-minimal surface with flat normal connection has
pointwise 1-type Gauss map if and only if it has parallel mean
curvature vector field  (Theorem \ref{THMNonFlat}).
We give necessary and sufficient conditions for a quasi-minimal surface with non-flat normal connection to have pointwise 1-type Gauss map.
In  Theorem \ref{THMNONFLATNORMALMainTheo} we present the complete classification of quasi-minimal surface with non-flat normal connection and pointwise 1-type Gauss map.
At the end of the paper we give an explicit example of a quasi-minimal surface with non-flat normal connection and pointwise 1-type Gauss map. This is also an example of a quasi-minimal surface with non-parallel mean curvature vector field.

\section{Preliminaries}

Let $\E^m_s$ be the pseudo-Euclidean $m$-space endowed with the
canonical pseudo-Euclidean metric of index $s$ given by
$$g_0 =  \sum_{i=1}^{m-s} dx_i^2 - \sum_{j=m-s+1}^{m} dx_j^2,$$
where $x_1, x_2, \hdots, x_m$  are rectangular coordinates of the
points of $\mathbb E^m_s$. As usual, we denote by $\langle \, , \rangle$ the
indefinite inner scalar product with respect to $g_0$.

A non-zero vector $v$ is said to be  \emph{spacelike} (respectively, \emph{timelike}) if $\langle v, v \rangle > 0$ (respectively, $\langle v, v \rangle < 0$).
A vector $v$ is called \emph{lightlike} if it is nonzero and satisfies $\langle v, v \rangle = 0$.

We use the following denotations:
 $$\mathbb{S}^{m-1}_s(1) = \{ v \in \E^m_s: \langle v,v \rangle = 1\},$$
$$\mathbb{H}^{m-1}_{s-1}(-1) = \{ v \in \E^m_s: \langle v,v \rangle = - 1\}.$$
$\mathbb{S}^{m-1}_s(1)$ and $\mathbb{H}^{m-1}_{s-1}(-1)$ ($m\geq 3$) are complete pseudo-Riemannian manifolds with constant sectional curvatures
$1$ and $-1$, respectively.
The pseudo-Euclidean space $\E^m_1$ is known as the \emph{Minkowski $m$-space},
the space $\mathbb{S}^{m-1}_1(1)$ is known as the \emph{de Sitter space}, and the space  $\mathbb{H}^{m-1}_1(-1)$ is the \emph{hyperbolic space}
(or the \emph{anti-de Sitter space}) \cite{O'N}.

\vskip 1mm The Gauss map $G$ of a submanifold $M^n$ of
$\E^m_s$ is defined as follows. Let $G(n,m)$ be the Grassmannian
manifold consisting of all oriented $n$-planes through the origin
of $\mathbb{E}^m_s$ and $\wedge ^{n}\mathbb{E}^m_s$ be the vector
space obtained by the exterior product of $n$ vectors in
$\mathbb{E}^{m}_s$.
Let $e_{i_1} \wedge \dots \wedge e_{i_n}$ and $f_{j_1} \wedge \dots \wedge f_{j_n}$ be two vectors of $\wedge^{n}\mathbb{E}^{m}_s$.
 The indefinite
inner product on the Grassmannian manifold is defined by
\begin{equation}\notag
\langle e_{i_1} \wedge \dots \wedge e_{i_n}, f_{j_1} \wedge \dots \wedge f_{j_n}  \rangle =
\det \left( \langle e_{i_k}, f_{j_l}  \rangle  \right).
\end{equation}
Thus, in a natural way, we can identify $\wedge
^{n}\mathbb{E}^{m}_s$ with some pseudo-Euclidean space $\mathbb{E}^{N}_k$,
where $N=\left(
\begin{array}{c}
m \\
n
\end{array}
\right)$, and $k$ is a positive integer.
 Let $\left\{ e_{1},...,e_{n},e_{n+1},\dots,e_{m}\right\} $ be a
 local orthonormal frame field in $\mathbb{E}^{m}_s$ such that $e_{1},e_{2},\dots,$ $e_{n}$ are tangent to $M^n$ and
 $e_{n+1},e_{n+2},\dots,e_{m}$ are
 normal to $M^n$.
The map $G:M^n \rightarrow G(n,m)$ defined by $%
G(p)=(e_{1}\wedge e_{2}\wedge \dots \wedge $ $e_{n})(p)$ is called the \emph{Gauss
map} of $M^n$. It is a smooth map which carries a point $p$ of $M^n$ into the
oriented $n$-plane in $\mathbb{E}^{m}_s$ obtained by the parallel translation
of the tangent space of $M^n$ at $p$ in $\mathbb{E}^{m}$ \cite{KY2-a}.
See also \cite{Turg} for detailed information about definition and geometrical interpretation of the Gauss map of submanifolds.

For any real valued  function $\varphi$ on $M^n$ the Laplacian $\Delta \varphi$ of $\varphi$ is given by the formula
\begin{equation}\notag
\Delta \varphi =-\sum_{i} \varepsilon_i (\widetilde\nabla_{e_{i}}\widetilde\nabla_{e_{i}}\varphi -\widetilde\nabla_{\widetilde\nabla_{e_{i}}e_{i}}\varphi ),
\end{equation}
where $\varepsilon_i = \langle e_i \wedge e_i \rangle = \pm 1$, $\widetilde\nabla$ is the Levi-Civita connection of $\E^m_s$  and $\nabla$ is the induced connection on $M^n$.

\vskip 2mm
In the present paper we consider the  pseudo-Euclidean 4-dimensional space $\E^4_2$ with the canonical pseudo-Euclidean metric of index 2. In this case, the metric $g_0$ becomes $g_0 = dx_1^2 + dx_2^2 - dx_3^2 - dx_4^2$.
A surface $M^2_1$ in $\E^4_2$ is called \emph{Lorentz}  if the
induced  metric $g$ on $M^2_1$ is Lorentzian. So, at each point $p\in M^2_1$ we have the following decomposition
$$\E^4_2 = T_pM^2_1 \oplus N_pM^2_1$$
with the property that the restriction of the metric
onto the tangent space $T_pM^2_1$ is of
signature $(1,1)$, and the restriction of the metric onto the normal space $N_pM^2_1$ is of signature $(1,1)$.

We denote by $\nabla$ and $\widetilde\nabla$ the Levi Civita connections of $M^2_1$  and $\E^4_2$, respectively.
For vector fields $X$, $Y$  tangent to $M^2_1$ and a vector field $\xi$ normal to $M^2_1$, the formulas of Gauss and Weingarten,
giving a decomposition of the vector fields $\widetilde\nabla_XY$ and
$\widetilde\nabla_X \xi$ into tangent and  normal components, are given respectively by \cite{Chen1}:
$$\begin{array}{l}
\vspace{2mm}
\widetilde\nabla_XY = \nabla_XY + h(X,Y);\\
\vspace{2mm}
\widetilde\nabla_X \xi = - A_{\xi} X + D_X \xi.
\end{array}$$

These formulas define the second fundamental form $h$, the normal
connection $D$,  and the shape operator $A_{\xi}$ with respect to
$\xi$. For each normal vector field $\xi$, the shape operator $A_{\xi}$ is a symmetric endomorphism of the tangent space $T_pM^2_1$ at $p \in M^2_1$. In general, $A_{\xi}$ is not diagonalizable.
It is well known that the shape operator and the second fundamental form are related by the formula
$$\langle h(X,Y), \xi \rangle = \langle A_{\xi} X, Y \rangle$$
for $X$, $Y$ tangent to $M^2_1$ and $\xi$ normal to $M^2_1$.

The mean curvature vector  field $H$ of $M^2_1$ in $\E^4_2$
is defined as $H = \ds{\frac{1}{2}\,  \tr\, h}$.
The  surface $M^2_1$  is called \emph{minimal} if its mean curvature vector vanishes identically, i.e. $H =0$.
The surface $M^2_1$  is called \emph{quasi-minimal} if its
mean curvature vector is lightlike at each point, i.e. $H \neq 0$ and $\langle H, H \rangle =0$.
Obviously, quasi-minimal surfaces are always non-minimal.

A normal vector field $\xi$ on $M^2_1$ is called \emph{parallel in the normal bundle} (or simply \emph{parallel}) if $D{\xi}=0$ holds identically \cite{Chen2}.
The surface $M^2_1$ is said to have \emph{parallel mean curvature vector field} if its mean curvature vector $H$
satisfies $D H =0$ identically.

\section{Classification of quasi-minimal surfaces with pointwise 1-type Gauss map} \label{S:Classification}

In this section we study quasi-minimal surfaces in pseudo-Euclidean space $\E^4_2$. We obtain complete classification of  quasi-minimal surfaces  with pointwise 1-type Gauss map.

\subsection{Moving frame on a quasi-minimal surface} \label{SubSect:ConForm}

Let $M^2_1$ be a Lorentz surface in $\E^4_2$. Then, locally there exists a  coordinate system $(u,v)$ on $M^2_1$ such that the  metric tensor  is given by
$$g=-f^2(u,v)(du \otimes dv +dv \otimes du)$$ for some positive function $f(u,v)$ \cite{Chen2}.
Thus,  putting $x=f^{-1}\frac{\partial}{\partial u}$ and $y=f^{-1}\frac{\partial}{\partial v}$, we obtain a pseudo-orthonormal frame field $\{x, y\}$ of the tangent bundle of $M^2_1$ such that $\langle x, x\rangle = 0$, $\langle y, y\rangle = 0$, $\langle x, y\rangle = -1$.
Then the mean curvature vector field $H$ is given by
$$H = - h(x,y).$$

Now, let $M^2_1$ be  quasi-minimal, i.e. its mean curvature vector is lightlike at each point. Then there exists a pseudo-orthonormal frame field  $\{n_1,n_2\}$ of the normal bundle   such that $n_1 = -H$,  $\langle n_1, n_1 \rangle = 0$, $\langle n_2, n_2 \rangle = 0$, $\langle n_1, n_2 \rangle = -1$.

By a direct computation we obtain the following derivative formulas:
\begin{subequations}\label{LeviCivitaConnection1ALL}
\begin{eqnarray}
\label{LeviCivitaConnection1a} \widetilde\nabla_xx=\gamma_1x+an_1+bn_2, & \qquad &\widetilde\nabla_yx=-\gamma_2x+n_1,\\
\label{LeviCivitaConnection1b} \widetilde\nabla_xy=-\gamma_1y+n_1, &\qquad &\widetilde\nabla_yy=\gamma_2y+cn_1+dn_2,\\
\label{LeviCivitaConnection1c} \widetilde\nabla_xn_1=-by+\beta_1n_1, &\qquad &\widetilde\nabla_yn_1=-dx+\beta_2n_1,\\
\label{LeviCivitaConnection1d} \widetilde\nabla_xn_2=-x-ay-\beta_1n_2, &\qquad &\widetilde\nabla_yn_2=-cx-y-\beta_2n_2
\end{eqnarray}
\end{subequations}
for some smooth functions $a,b,c,d$, $\beta_1$, and $\beta_2$, where $\gamma_1=\frac{f_u}{f^2}$ and $\gamma_2=\frac{f_v}{f^2}.$
Thus, the Gaussian curvature $K$ and the normal curvature $\varkappa$ of $M^2_1$ are
\begin{eqnarray}
\label{GaussianEQ} K = -R(x,y,y,x)= x(\gamma_2) + y(\gamma_1) +2\gamma_1 \gamma_2, \notag\\
\label{kappaeq} \varkappa = -R^D(x,y,n_1,n_2)= x(\beta_2) - y(\beta_1) + \gamma_1 \beta_2 - \gamma_2 \beta_1,
\end{eqnarray}
where  $R$ and $R^D$ are the curvature tensors associated with the connections $\nabla$  and $D$, respectively.

\begin{lem}\label{Lemma-parallel H}
If $M^2_1$ is a quasi-minimal surface with parallel mean curvature vector field, then $M^2_1$ has flat normal connection.
\end{lem}

\begin{proof}
It follows from  \eqref{LeviCivitaConnection1c} that $D_xH=\beta_1H$, $D_yH=\beta_2H$. Hence, the mean curvature vector field $H$ is parallel if and only if $\beta_1 = \beta_2 = 0$. Now, under the assumption  $\beta_1 = \beta_2 = 0$,  equality \eqref{kappaeq} implies $\varkappa = 0$. Therefore, if $H$ is parallel then the surface has flat normal connection.

\end{proof}

Using the  equations of Gauss and Codazzi, from formulas  \eqref{LeviCivitaConnection1ALL}
we obtain the following integrability conditions:
\begin{subequations}\label{IntEqAll}
\begin{eqnarray}
\label{IntEq4} x(c)=-c\beta_1-2c\gamma_1+\beta_2,\\
\label{IntEq5} x(d)=d\beta_1-2d\gamma_1,\\
\label{IntEq2} y(a)=-a\beta_2-2a\gamma_2+\beta_1,\\
\label{IntEq3} y(b)=b\beta_2-2b\gamma_2,\\
\label{IntEq1} x(\gamma_2) + y(\gamma_1) + 2\gamma_1\gamma_2 = ad +bc,\\
\label{IntEq6} x(\beta_2) - y(\beta_1) - \beta_1 \gamma_2 + \beta_2 \gamma_1  = ad - bc.
\end{eqnarray}
\end{subequations}

Equalities  \eqref{IntEq1} and \eqref{IntEq6} imply that the Gauss curvature $K$ and the normal curvature $\varkappa$ are expressed as follows:
\begin{eqnarray}
\label{GaussianEQ-a} K = ad + bc,\\
\label{kappaeq-a} \varkappa = ad - bc.
\end{eqnarray}

\begin{rem}\label{Casebd0}
If $b=d=0$, then \eqref{GaussianEQ-a} implies $K=0$, i.e. $M^2_1$ is flat. In \cite[Theorem 4.1]{Chen-JMAA}, B.-Y. Chen obtained complete classification of flat quasi-minimal surfaces in  $\E^4_2$. Considering the proof of this theorem, one can see that a quasi-minimal surface in $\E^4_2$ satisfying $b=d=0$ is congruent to a surface parametrized by
\begin{equation}\label{PropSpecSurfPosVec}
z(u,v)=\left(\theta(u,v),\frac{u-v}{\sqrt2},\frac{u+v}{\sqrt2},\theta(u,v)\right)
\end{equation}
for a smooth function $\theta$.
\end{rem}

\begin{lem}\label{LemmaConstGauss}
Let $M^2_1$ be a quasi-minimal surface in  $\mathbb E^4_2$ with parallel mean curvature vector field. If $M^2_1$ has constant Gauss curvature, then  $M^2_1$ is flat.
\end{lem}

\begin{proof}
Let $M^2_1$ be a quasi-minimal surface with parallel mean curvature vector field and constant Gaussian curvature $K_0$.
Assume that $K_0 \neq 0$. Since  $M^2_1$ has parallel mean curvature vector, according to Lemma
\ref{Lemma-parallel H} it has flat normal connection. Therefore, because of \eqref{kappaeq-a}, we have
$ad = bc$ and hence the Gauss curvature is $K_0 = 2 ad$. Since $K_0 \neq 0$ we have that
 $a,b,c,d$ do not vanish and satisfy
\begin{equation}\label{LemmaConstGaussEq2}
a=\alpha c, \quad b=\alpha  d
\end{equation}
for a smooth function $\alpha(u,v)$. Note that  $\alpha$ does not vanish and
$\alpha cd=\mbox{const}$.

In addition, since  $M^2_1$ has parallel mean curvature vector field, we have $\beta_1=\beta_2=0$. Therefore, equalities \eqref{IntEq2}, \eqref{IntEq4}, \eqref{IntEq5} take the form
\begin{subequations}\label{IntEqaaaAll}
\begin{eqnarray}
\label{IntEq4a} x(c)=-2c\gamma_1,\\
\label{IntEq5a} x(d)=-2d\gamma_1,\\
\label{IntEq2a} y(a)=-2a\gamma_2.
\end{eqnarray}
\end{subequations}
Equalities \eqref{IntEq4a}, \eqref{IntEq5a} together with \eqref{LemmaConstGaussEq2} imply that
\begin{equation}\label{LemmaConstGaussEq3}
x(\alpha)=4\alpha\gamma_1.
\end{equation}
Applying $x$ to the first equality of \eqref{LemmaConstGaussEq2} and using \eqref{IntEq4a} and \eqref{LemmaConstGaussEq3}
  we obtain $x(a)=2 a \gamma_1$. Having in mind that $\gamma_1 = \frac{f_u}{f^2}$, we get
\begin{equation*}\label{LemmaConstGaussEq4}
\frac{a_u}{a}=2\frac{f_u}{f},
\end{equation*}
which implies $a(u,v)= \varphi(v) f^2(u,v)$ for a smooth function $\varphi(v)$. Now using \eqref{IntEq2a} we get
$\displaystyle \frac{\varphi'(v)}{4\varphi(v)}= -\frac{f_v}{f}$. Since $\varphi$ is a function of  $v$, we obtain $\displaystyle \frac{\partial}{\partial u}\left(\frac{f_v}{f}\right)=0$. Solving this equation, we get $f(u,v)=f_1(u) f_2(v)$ for some smooth functions $f_1(u), \ f_2(v)$. On the other hand, the Gauss curvature is expressed by the function $f(u,v)$ according to the formula
$K =\frac{2f f_{uv} - 2 f_u f_v}{f^4}$. Hence, $K_0=0$, which contradicts the assumption $K_0 \neq 0$.

\end{proof}

\subsection{Gauss map of quasi-minimal surfaces} \label{SubSect:MainPARTY}

Let $M^2_1$ be a quasi-minimal surface in $\E^4_2$. The Gauss map of $M^2_1$ is defined by
\begin{equation*}\label{MinkGaussTasvTanim}
\begin{array}{rcl}
G: M &\rightarrow & \mathbb H^{5}_3(-1)\subset \mathbb E^6_4\\
p & \mapsto & G(p)=(x\wedge y)(p).
\end{array}
\end{equation*}
We shall use the frame field $\{x,y,n_1,n_2\}$ defined in the previous subsection. This frame field generates the following frame of the Grassmanian manifold:
$$\{x \wedge y, x \wedge n_1, x \wedge n_2,
y \wedge n_1, y \wedge n_2, n_1 \wedge n_2\},$$
for which we have
\begin{equation} \notag
\begin{array}{lll}
\vspace{2mm}
\langle x \wedge y, x \wedge y  \rangle = -1, & \qquad \langle x \wedge n_1, x \wedge n_1  \rangle = 0, & \qquad \langle x \wedge n_2, x \wedge n_2  \rangle = 0,\\
\vspace{2mm}
\langle y \wedge n_1, y \wedge n_1  \rangle = 0, & \qquad \langle y \wedge n_2, y \wedge n_2  \rangle = 0, & \qquad \langle n_1 \wedge n_2, n_1 \wedge n_2  \rangle = - 1,\\
\langle x \wedge n_1, y\wedge n_2  \rangle = 1, & \qquad \langle x \wedge n_2, y \wedge n_1  \rangle = 1, &
\end{array}
\end{equation}
and all other scalar products are equal to zero.

Since $\langle x, x\rangle = \langle y, y\rangle = 0, \langle x, y\rangle = -1$,
the Laplace operator $\Delta:C^\infty(M^2_1)\rightarrow C^\infty(M^2_1)$ of $M^2_1$ takes the form
$$\Delta \varphi = \widetilde\nabla_x \widetilde\nabla_y \varphi + \widetilde\nabla_y \widetilde\nabla_x \varphi - \widetilde\nabla_{\nabla_x y} \varphi - \widetilde\nabla_{\nabla_y x} \varphi$$
for any real valued function $\varphi$.

Hence, the Laplacian of the Gauss map is given by the
formula
\begin{equation*}\label{Eq-7}
\Delta G = \widetilde\nabla_x \widetilde\nabla_y G + \widetilde\nabla_y \widetilde\nabla_x G - \widetilde\nabla_{\nabla_x y} G - \widetilde\nabla_{\nabla_y x} G.
\end{equation*}
By a direct computation, using \eqref{LeviCivitaConnection1ALL}, \eqref{IntEqAll}, \eqref{GaussianEQ-a}, and \eqref{kappaeq-a}, we obtain
\begin{equation}\label{E42MargTrapGMLap}
\Delta G = -2 K x\wedge y + 2\varkappa \,n_1\wedge n_2 + 2\beta_2 x\wedge n_1 - 2\beta_1  y\wedge n_1.
\end{equation}

The next proposition follows directly from formula \eqref{E42MargTrapGMLap}.

\begin{prop}\label{THMGMHarmonic}
Let $M^2_1$ be a quasi-minimal surface in the pseudo-Euclidean space $\E^4_2$. Then, $M^2_1$ has harmonic Gauss map if and only if  $M^2_1$  is a  flat surface with  parallel mean curvature vector field.
\end{prop}

\begin{rem}
See  \cite{HouYang2010} for the classification of Lorentzian surfaces with parallel mean curvature vector in $\mathbb E^4_2$.
\end{rem}

Further we shall study quasi-minimal surfaces with pointwise 1-type Gauss map, i.e. the Laplacian of $G$ satisfies \eqref{PW1typeEq1} for a smooth non-vanishing function $\phi$ and a constant vector $C$.

Having in mind Lemma \ref{Lemma-parallel H} and Lemma \ref{LemmaConstGauss}, from \eqref{E42MargTrapGMLap} we have the following proposition.

\begin{prop}\label{THMNonFlat1type}
Let $M^2_1$ be a quasi-minimal surface in the pseudo-Euclidean space $\E^4_2$. If $M^2_1$ is a non-flat surface with parallel mean curvature vector field, then it has proper pointwise 1-type Gauss map of first kind. In this case, \eqref{PW1typeEq1} is satisfied for the smooth function $\phi=-2K$.
\end{prop}

Now, we shall give the complete classification of quasi-minimal surfaces with pointwise 1-type Gauss map.

Assume that $M^2_1$ has pointwise 1-type Gauss map. Then from \eqref{PW1typeEq1} and \eqref{E42MargTrapGMLap} we get the equality
$$-2Kx\wedge y+2\varkappa n_1\wedge n_2+2\beta_2  x\wedge n_1-2\beta_1 y\wedge n_1=\phi(G+C), \quad \phi \neq 0,$$
which implies
\begin{subequations}\label{PW1TypeEq1ALL}
\begin{eqnarray}
\label{PW1TypeEq1a}\langle C, x\wedge y\rangle&=&1+\frac{2K}\phi,\\
\label{PW1TypeEq1b}\langle C,  n_1\wedge n_2\rangle&=&-\frac{2\varkappa}\phi,\\
\label{PW1TypeEq1c}\langle C, x\wedge n_1\rangle&=&0,\\
\label{PW1TypeEq1d}\langle C, y\wedge n_1\rangle&=&0,\\
\label{PW1TypeEq1e}\langle C, x\wedge n_2\rangle&=&-\frac{2\beta_1}\phi,\\
\label{PW1TypeEq1f}\langle C, y\wedge n_2\rangle&=&\frac{2\beta_2}\phi.
\end{eqnarray}
\end{subequations}

%%%%%%%%%%%%%%%%%%%%%%%%%%%%%%%%%%%%%%%%%%%%%%%%%%%%%%%%%%%%%%%%%%%%%%%%%%%%%%%%%%%%%%%%%%%%%%%%%%%%%%%%%%%%%%%%%%%%%%%%%%%%%%%%%%%%%%%%%%%%%%%%%%%%%%%%%%%%%%%%%%%%%%%%%%%%%%%%%%%%%%%%%%%%%%%%%%%%%%%%%%%%%%%%%%%%%%%%%%%%%%%%%%%%%%%%%%%%%%%%%%%%%%%%%%%%%%%%%%%%%%%%%%%%%%%%%%%%%%%%%%%%%%%%%%%%%%%%%%%%%%%%%%%%%%%%%%%%%%%%%%%%%%%%%%%%%%%%%%%%%%%%%%%%%%%%%%%%%%%%%%%%%%%%%%%%%%%%%%%%%%%%%%%%%%%%%%%%%%%%%%%%%%%%%%%%%%%%%%%%%%%%%%%%%%%%%%%%%%%%%%%%%%%%%%%%%%%%%%%%%%%%%%%%%%%%%%%%%%%%%%%%%%%%%%%%%%%%%%%%%%%%%%%%%%%%%%%%%%%%%%%%%%%%%%%%%%%%%%%%%%%%%%%%%%%%%%%%%%%%%%%%%%%%%%%%%%%%%%%%%%%%%%%%%%%%%%%%%%%%%%%%%%%%%%%%%%%%%%%%%%%%%%%%%%%%%%%%%%%%%%%%%%%%%%%%%%%%%%%%%%%%%%%%%%%%%%%%%%%%%%%%%%%%%%%%%%%%%%%%%%%%%%%%%%%%%%%%%%%%%%%%%%%%%%%%%%%%%%%%%%%%%%%%%%%%%%%%%%%%%%%%%%%%%%%%%%%%%%%%%%%%%%%%%%%%%%%%%%%%%%%%%%%%%%%%%%%%%%%%%%%%%%%%%%%%%%%%%%%%%%%%%%%%%%%%%%%%%%%%%%%%%%%%%%%%%%%%%%%%%%%%%%%%%%%%%%%%%%%%%%%%%%%%%%%%%%%%%%%%%%%%%%%%%%%%%%%%%%%%%%%%%%%%%%%%%%%%%%%%%%%%%%%%%%%%%%%%%%%%%%%%%%%%%%%%%%%%%%%%%%%%%%%%%%%%%%%%%%%%%%%%%%%%%%%%%%%%%%%%

\subsection{Flat quasi-minimal surfaces with pointwise 1-type Gauss map}

In this subsection we give the classification of flat quasi-minimal surfaces in $\E^4_2$ with  pointwise 1-type Gauss map.

Let $M^2_1$ be a flat quasi-minimal surface  with pointwise 1-type Gauss map, i.e., \eqref{PW1typeEq1} is satisfied for a non-zero function $\phi$ and a constant vector $C$. Applying $x$ and $y$ to \eqref{PW1TypeEq1a} and using $K = 0$, we obtain
$\beta_2 b = \beta_1 d = 0$.

Note that if $M^2_1$ has parallel mean curvature vector, then according to Proposition \ref{THMGMHarmonic} the Gauss map is harmonic. Since we consider the non-harmonic case, we have $\beta_1^2+\beta_2^2\neq0$. Hence, $b \,d=0$, i.e. at least one of the functions $b$ or $d$ is zero. If we assume that $b = 0, d \neq 0$, then we get $\varkappa=0, \beta_1 = 0, \beta_2 = 0$. Applying $y$ to \eqref{PW1TypeEq1e} we obtain $-1 = 0$, which is a contradiction. Similarly for the case $d = 0, b \neq 0$. Hence, the only possible case is $b = d =0$.
Thus,  according to Remark 3.2 $M^2_1$ is congruent to the surface given by \eqref{PropSpecSurfPosVec} for some function $\theta(u,v)$. Now, using the condition that the surface has pointwise 1-type Gauss map we will obtain conditions on  $\theta(u,v)$.

We put $\eta_0=\left(1,0,0,1\right)$, $\displaystyle \eta_1=\left(0,\frac{1}{\sqrt2},\frac{1}{\sqrt2},0\right)$ and $\displaystyle \eta_2=\left(0,-\frac{1}{\sqrt2},\frac{1}{\sqrt2},0\right)$.

\begin{thm}\label{PropFlatQuasi}
Let $M^2_1$ be a flat quasi-minimal surface  in the pseudo-Euclidean space $\mathbb E^4_2$. Then, $M^2_1$ has pointwise 1-type Gauss map if and only if it is congruent to the surface given by \eqref{PropSpecSurfPosVec} for a smooth function $\theta$ satisfying
\begin{equation}\label{PropSpecSurfSufNecEq}
\frac{\partial^2\theta(u,v)}{\partial u\partial v}=(F\circ\psi)(u,v), \quad\psi(u,v)=\theta(u,v)+c_1u+c_2v,
\end{equation}
where $F$ is a non-constant function,  $c_1$ and $c_2$ are constants. In this case, \eqref{PW1typeEq1} is satisfied for the smooth function
\begin{equation}\label{PropSpecSurfFuncPHI}
\phi=(F'\circ\psi)
\end{equation}
and the non-zero constant vector
\begin{equation}\label{PropSpecSurfC}
C=c_1\eta_0\wedge\eta_2-c_2\eta_0\wedge\eta_1-\eta_1\wedge\eta_2.
\end{equation}
\end{thm}

\begin{proof}
By a direct computation from \eqref{PropSpecSurfPosVec}  we get
\begin{align}\nonumber
\begin{split}
x=&z_u=\theta_u\eta_0+\eta_1,\\
y=&z_v=\theta_v\eta_0+\eta_2,
\end{split}
\end{align}
which imply
\begin{equation}\label{PropSpecSurfG}
G=\eta_0\wedge(\theta_u\eta_2-\theta_v\eta_1)+\eta_1\wedge\eta_2
\end{equation}
and
\begin{equation}\label{PropSpecSurfDG}
\Delta G=\eta_0\wedge(2\zeta_u\eta_2-2\zeta_v\eta_1),
\end{equation}
where we put $\zeta=\frac{\partial^2\theta}{\partial u\partial v}.$

From  \eqref{PW1typeEq1},  \eqref{PropSpecSurfG}  and  \eqref{PropSpecSurfDG}  we obtain
\begin{align}\nonumber
\begin{split}
\eta_0\wedge\left(\left(2\frac{\zeta_u}{\phi}-\theta_u\right)\eta_2-\left(2\frac{\zeta_v}{\phi}-\theta_v\right)\eta_1\right)=C+\eta_1\wedge\eta_2.
\end{split}
\end{align}
Since the right hand side of this equation is a constant vector, the vector field
$$\eta_0\wedge\left(\left(2\frac{\zeta_u}{\phi}-\theta_u\right)\eta_2-\left(2\frac{\zeta_v}{\phi}-\theta_v\right)\eta_1\right)$$
is also constant. So, we obtain that the  system of equations
\begin{align}\nonumber
\begin{split}
2\frac{\zeta_u}{\phi}-\theta_u = & c_1,\\
2\frac{\zeta_v}{\phi}-\theta_v = & c_2
\end{split}
\end{align}
is satisfied for some constants $c_1$ and $c_2$. The last two equations imply
\begin{align*}\label{PropSpecSurfEq1}
\phi =2\frac{\zeta_u}{c_1+\theta_u}=2\frac{\zeta_v}{c_2+\theta_v}.
\end{align*}
From this equation, one can see that the function $\zeta$ remains constant along the curve $\psi=c$, where $\psi$ is the function given by the second equality in \eqref{PropSpecSurfSufNecEq}. Hence, we have proved that $\theta$ satisfies the first equality  in \eqref{PropSpecSurfSufNecEq}.

Conversely, by a straightforward computation one can see that  \eqref{PW1typeEq1} is satisfied for the constant vector $C$ and the smooth non-zero function $F$ given by \eqref{PropSpecSurfFuncPHI} and \eqref{PropSpecSurfC}.
\end{proof}

%%%%%%%%%%%%%%%%%%%%%%%%%%%%%%%%%%%%%%%%%%%%%%%%%%%%%%%%%%%%%%%%%%%%%%%%%%%%%%%%%%%%%%%%%%%%%%%%%%%%%%%%%%%%%%%%%%%%%%%%%%%%%%%%%%%%%%%%%%%%%%%%%%%%%%%%%%%%%%%%%%%%%%%%%%%%%%%%%%%%%%%%%%%%%%%%%%%%%%%%%%%%%%%%%%%%%%%%%%%%%%%%%%%%%%%%%%%%%%%%%%%%%%%%%%%%%%%%%%%%%%%%%%%%%%%%%%%%%%%%%%%%%%%%%%%%%%%%%%%%%%%%%%%%%%%%%%%%%%%%%%%%%%%%%%%%%%%%%%%%%%%%%%%%%%%%%%%%%%%%%%%%%%%%%%%%%%%%%%%%%%%%%%%%%%%%%%%%%%%%%%%%%%%%%%%%%%%%%%%%%%%%%%%%%%%%%%%%%%%%%%%%%%%%%%%%%%%%%%%%%%%%%%%%%%%%%%%%%%%%%%%%%%%%%%%%%%%%%%%%%%%%%%%%%%%%%%%%%%%%%%%%%%%%%%%%%%%%%%%%%%%%%%%%%%%%%%%%%%%%%%%%%%%%%%%%%%%%%%%%%%%%%%%%%%%%%%%%%%%%%%%%%%%%%%%%%%%%%%%%%%%%%%%%%%%%%%%%%%%%%%%%%%%%%%%%%%%%%%%%%%%%%%%%%%%%%%%%%%%%%%%%%%%%%%%%%%%%%%%%%%%%%%%%%%%%%%%%%%%%%%%%%%%%%%%%%%%%%%%%%%%%%%%%%%%%%%%%%%%%%%%%%%%%%%%%%%%%%%%%%%%%%%%%%%%%%%%%%%%%%%%%%%%%%%%%%%%%%%%%%%%%%%%%%%%%%%%%%%%%%%%%%%%%%%%%%%%%%%%%%%%%%%%%%%%%%%%%%%%%%%%%%%%%%%%%%%%%%%%%%%%%%%%%%%%%%%%%%%%%%%%%%%%%%%%%%%%%%%%%%%%%%%%%%%%%%%%%%%%%%%%%%%%%%%%%%%%%%%%%%%%%%%%%%%%%%%%%%%%%%%%%%%%%%%%%%%%%%%%%%%%%%%%%%%%%%%%%%%%%

\subsection{Quasi-minimal surfaces with flat normal connection}

In this subsection, we focus on quasi-minimal surfaces with $\varkappa=0$.

Let $M^2_1$ be a quasi-minimal surface with flat normal connection and pointwise 1-type Gauss map. Then, \eqref{PW1TypeEq1b} implies $\langle C,  n_1\wedge n_2\rangle=0$.  Applying $x$ and $y$ to the last equation and using \eqref{PW1TypeEq1e} and \eqref{PW1TypeEq1f} we obtain
\begin{equation*}\label{PW1TypeLastEq}
b\beta_2=d\beta_1=0.
\end{equation*}
On the other hand, from \eqref{GaussianEQ-a} and \eqref{kappaeq-a} we have $K=2bc$.
If the Gauss curvature  does not vanish, then $\beta_1 = \beta_2 = 0$ and hence $M^2_1$ has parallel mean curvature vector field. Combining this with Proposition \ref{THMNonFlat1type}, we get the following result.

\begin{thm}\label{THMNonFlat}
Let $M^2_1$ be a quasi-minimal surface in the pseudo-Euclidean space $\mathbb E^4_2$ with flat normal connection and non-vanishing Gauss curvature. Then, $M^2_1$ has pointwise 1-type Gauss map if and only if it has parallel mean curvature vector field. In this case, $M^2_1$ has proper pointwise 1-type Gauss map of  first kind.
\end{thm}

\begin{rem}
We would like to note that a classification of Lorentz surfaces with parallel mean curvature vector field in the pseudo-Euclidean space $\mathbb E^4_2$ is given in \cite[Theorem 3.1]{HouYang2010}. Considering this theorem and its proof, one can see that there exist two families of quasi-minimal surfaces with proper pointwise 1-type Gauss map of first kind:
\begin{enumerate}
\item[(i)] A non-flat CMC-surface lying in $\mathbb S^3_2(r^2)$ for some $r>0$ such that the
mean curvature vector $H'$ of $M$ in $\mathbb S^3_2(r^2)$ satisfies $\langle H',H'\rangle=-r^2$;
\item[(ii)] A non-flat CMC-surface lying in $\mathbb H^3_1(-r^2)$ for some $r>0$ such that the
mean curvature vector $H'$ of $M$ in $\mathbb H^3_2(-r^2)$ satisfies $\langle H',H'\rangle=r^2$.
\end{enumerate}
Conversely, any quasi-minimal surface with proper pointwise 1-type Gauss map of first kind belongs to one of the above two families.
\end{rem}
%%%%%%%%%%%%%%%%%%%%%%%%%%%%%%%%%%%%%%%%%%%%%%%%%%%%%%%%%%%%%%%%%%%%%%%%%%%%%%%%%%%%%%%%%%%%%%%%%%%%%%%%%%%%%%%%%%%%%%%%%%%%%%%%%%%%%%%%%%%%%%%%%%%%%%%%%%%%%%%%%%%%%%%%%%%%%%%%%%%%%%%%%%%%%%%%%%%%%%%%%%%%%%%%%%%%%%%%%%%%%%%%%%%%%%%%%%%%%%%%%%%%%%%%%%%%%%%%%%%%%%%%%%%%%%%%%%%%%%%%%%%%%%%%%%%%%%%%%%%%%%%%%%%%%%%%%%%%%%%%%%%%%%%%%%%%%%%%%%%%%%%%%%%%%%%%%%%%%%%%%%%%%%%%%%%%%%%%%%%%%%%%%%%%%%%%%%%%%%%%%%%%%%%%%%%%%%%%%%%%%%%%%%%%%%%%%%%%%%%%%%%%%%%%%%%%%%%%%%%%%%%%%%%%%%%%%%%%%%%%%%%%%%%%%%%%%%%%%%%%%%%%%%%%%%%%%%%%%%%%%%%%%%%%%%%%%%%%%%%%%%%%%%%%%%%%%%%%%%%%%%%%%%%%%%%%%%%%%%%%%%%%%%%%%%%%%%%%%%%%%%%%%%%%%%%%%%%%%%%%%%%%%%%%%%%%%%%%%%%%%%%%%%%%%%%%%%%%%%%%%%%%%%%%%%%%%%%%%%%%%%%%%%%%%%%%%%%%%%%%%%%%%%%%%%%%%%%%%%%%%%%%%%%%%%%%%%%%%%%%%%%%%%%%%%%%%%%%%%%%%%%%%%%%%%%%%%%%%%%%%%%%%%%%%%%%%%%%%%%%%%%%%%%%%%%%%%%%%%%%%%%%%%%%%%%%%%%%%%%%%%%%%%%%%%%%%%%%%%%%%%%%%%%%%%%%%%%%%%%%%%%%%%%%%%%%%%%%%%%%%%%%%%%%%%%%%%%%%%%%%%%%%%%%%%%%%%%%%%%%%%%%%%%%%%%%%%%%%%%%%%%%%%%%%%%%%%%%%%%%%%%%%%%%%%%%%%%%%%%%%%%%%%%%%%%%%%%%%%%%%%%%%%%%%%%%%%%%%%%%

\subsection{Quasi-minimal surfaces with non-flat normal connection}
In this subsection we focus on quasi-minimal surfaces with non-flat normal connection and pointwise 1-type Gauss map. Before we proceed, we would like to note that  recent results show that there are no such surfaces if the ambient space is $\mathbb E^4_1$ or $\mathbb S^4_1(1)$ (see \cite{Mil,Turg,Turg2}).

First we prove the following proposition.

\begin{prop}\label{THMNONFLATNORMAL}
Let $M^2_1$ be a quasi-minimal surface in $\mathbb E^4_2$  with non-vanishing normal curvature. Then, $M^2_1$ has pointwise 1-type Gauss map if and only if there exists a local coordinate system $(s,t)$ such that $\bar x=\partial_s$, $\bar y=\widetilde f\partial_s+\partial_t$, $n_1$ and $n_2$ form a pseudo-orthonormal frame field of the tangent bundle of $M^2_1$ and the Levi-Civita connection satisfies
\begin{subequations}\label{LeviCivitaConnectionnonflat1ALL}
\begin{eqnarray}
\label{LeviCivitaConnectionnonflat1a} \widetilde\nabla_{\bar x}\bar x=\widetilde an_1, & \qquad &\widetilde\nabla_{\bar y}{\bar x}=-\widetilde f_s{\bar x}+n_1,\\
\label{LeviCivitaConnectionnonflat1b} \widetilde\nabla_{\bar x}{\bar y}=n_1, &\qquad &\widetilde\nabla_{\bar y}\bar y=\widetilde f_s{\bar y}+\widetilde c n_1+\widetilde dn_2,\\
\label{LeviCivitaConnectionnonflat1c} \widetilde\nabla_{\bar x}n_1=0, &\qquad &\widetilde\nabla_{\bar y}n_1=-\widetilde d{\bar x}+\widetilde\beta_2n_1,\\
\label{LeviCivitaConnectionnonflat1d} \widetilde\nabla_{\bar x}n_2=-{\bar x}-\widetilde a{\bar y}, &\qquad & \widetilde\nabla_{\bar y}n_2=-\widetilde c{\bar x}-{\bar y}-\widetilde \beta_2n_2
\end{eqnarray}
\end{subequations}
where $\widetilde a$, $\widetilde c$, $\widetilde d$, $\widetilde \beta_2$ and $\widetilde f$ are  smooth functions given by
\begin{subequations}\label{LeviCivitaConnectionnonflatInvariants1ALL}
\begin{eqnarray}
\label{LeviCivitaConnectionnonflatInvariants1ad}  \widetilde a&=&\lambda_3\left(-\frac 23 \lambda_1s-\frac 23 \lambda_2\right)^{-3/2}, \qquad\qquad \widetilde d=\frac{1}{\lambda_3},\\
\label{LeviCivitaConnectionnonflatInvariants1cFinal}  \widetilde c&=& -\frac{9}{\lambda_1^2}\left(-\frac 23 \lambda_1s-\frac 23 \lambda_2\right)^{1/2},\\
\label{LeviCivitaConnectionnonflatInvariants1beta} \widetilde\beta_2&=&\frac{3}{\lambda_1}\left(-\frac 23 \lambda_1s-\frac 23 \lambda_2\right)^{-1/2},\\
\label{LeviCivitaConnectionnonflatInvariants1fFinal} \widetilde f&=&-\frac{9}{\lambda_1^2}\left(-\frac 23 \lambda_1s-\frac 23 \lambda_2\right)^{1/2}-\left(2\frac{\lambda_3'}{\lambda_3}-3\frac{\lambda_1'}{\lambda_1}\right)s-\frac{\lambda_2'}{\lambda_1}-2\frac {\lambda_2}{\lambda_1}\frac{\lambda_3'}{\lambda_3}+4\frac {\lambda_2}{\lambda_1}\frac{\lambda_1'}{\lambda_1}
\end{eqnarray}
\end{subequations}
for  some non-vanishing smooth functions $\lambda_1,\lambda_3$ and a smooth function $\lambda_2$.

In this case, $M^2_1$ has Gauss  curvature
\begin{eqnarray}
\label{NewCoordPW1Cond1Gaussian} K&=&\left(-\frac 23 \lambda_1s-\frac 23 \lambda_2\right)^{-3/2}.
\end{eqnarray}
Moreover, \eqref{PW1typeEq1} is satisfied for
\begin{equation}\label{TheoremnonflatEQUfC}
\phi=-4K \quad\mbox{ and}\quad C=-\frac12\left(\bar x \wedge \bar y +n_1\wedge n_2+\frac {\widetilde \beta_2}{K}\bar x\wedge n_1\right).
\end{equation}
\end{prop}

\begin{proof}

Assume that the normal curvature $\varkappa$ is non-vanishing and $M^2_1$ has pointwise 1-type Gauss map. Then, \eqref{PW1typeEq1} is satisfied for a constant vector $C$ and a smooth non-vanishing function $\phi$ such that formulas  \eqref{PW1TypeEq1ALL} hold true.
Applying $x$ to \eqref{PW1TypeEq1c} and $y$ to \eqref{PW1TypeEq1d}, we obtain
\begin{subequations}\label{PW1TypeEq2ALL}
\begin{eqnarray}
\label{PW1TypeEq2a}b(\langle C, x\wedge y\rangle+\langle C,  n_1\wedge n_2\rangle)&=&0,\\
\label{PW1TypeEq2b}d(\langle C, x\wedge y\rangle-\langle C,  n_1\wedge n_2\rangle)&=&0.
\end{eqnarray}
\end{subequations}
Note that if $b^2+d^2=0$ at a point $p$, then  \eqref{kappaeq-a} implies $\varkappa(p)=0$ which contradicts the assumption that $\varkappa$ is non-vanishing.
If both $b$ and $d$ are non-zero, then $\langle C, x\wedge y\rangle = \langle C,  n_1\wedge n_2\rangle = 0$ and from \eqref{PW1TypeEq1b} we get $\varkappa = 0$, a contradiction. Hence, $b = 0, d \neq 0$ or $b \neq 0, d = 0$.
Therefore, replacing $x$ and $y$ if necessary, we may assume that $d\neq 0$ and $b=0$. Thus, \eqref{PW1TypeEq2b} gives $\langle C, x\wedge y\rangle=\langle C,  n_1\wedge n_2\rangle$ and equations  \eqref{GaussianEQ-a}, \eqref{kappaeq-a} imply $K=\varkappa = ad$.  Therefore, \eqref{PW1TypeEq1a} and \eqref{PW1TypeEq1b} imply
 \begin{equation}\label{TheoremnonflatEq1}
\langle C, x\wedge y\rangle=\langle C,  n_1\wedge n_2\rangle=\frac 12\quad\mbox{ and }\quad \phi=-4K.
\end{equation}
On the other hand, from $y(\langle C, x\wedge y\rangle)=0$ we obtain $d \langle C, x\wedge n_2\rangle = 0$ which gives $\beta_1=0$ because of \eqref{PW1TypeEq1e}. Therefore, combining \eqref{PW1TypeEq1ALL} and \eqref{TheoremnonflatEq1}, we obtain
\begin{equation}\label{TheoremnonflatEQUfCPre}
\quad C=-\frac12\left( x \wedge y +n_1\wedge n_2+\frac {\beta_2}{K} x\wedge n_1\right).
\end{equation}

Next, we define a local coordinate system $(s,t)$ in the following way:
$$s=s(u,v)=\int_{u_0}^uf^2(\tau,v)d\tau,\quad t=v.$$
Note that we have
\begin{equation*}\label{MinimalcaseII8}
\partial_u=f^2\partial_s\quad\mbox{and}\quad \partial_v=\widetilde f\partial_s+\partial_t,
\end{equation*}
where $\widetilde f=\frac{\partial}{\partial v}\left(\int_{u_0}^u f^2(\tau,v)d\tau\right)$,
which give
$$\langle\partial_s,\partial_s\rangle=0,\quad \langle\partial_s,\partial_t\rangle=-1,\quad \langle\partial_t,\partial_t\rangle=2\widetilde f.$$
 Thus, the metric tensor of $M^2_1$ with respect to the new coordinate system $(s,t)$  takes the form
\begin{equation*}\label{SurfaceQMinimalInducedM}
g=-(ds \otimes dt+dt \otimes ds) + 2 \widetilde{f} dt \otimes dt.
\end{equation*}
It is easy to calculate that
\begin{subequations}\label{QMinimalS421LeviCivita1ALL}
\begin{eqnarray}
\label{QMinimalS421LeviCivita1a} \nabla_{\partial_s}\partial_s&=&0,\\
\label{QMinimalS421LeviCivita1b} \nabla_{\partial_s}\partial_t= \nabla_{\partial_t}\partial_s&=& -\widetilde f_s\partial_s,\\
\label{QMinimalS421LeviCivita1c} \nabla_{\partial_t}\partial_t&=& \widetilde f_s\partial_t+(2\widetilde f\widetilde f_s-\widetilde f_t)\partial_s.
\end{eqnarray}
\end{subequations}
Moreover, $\bar x = \partial_s =\frac 1fx$ and  $\bar y = \widetilde f\partial_s+\partial_t = fy$. Hence, $\{\bar x, \bar y \}$ form a pseudo-orthonormal frame field of the tangent bundle of $M^2_1$. Using \eqref{QMinimalS421LeviCivita1ALL} and taking into account that $b=\beta_1=0$, we get that the Levi-Civita connection of $M^2_1$ satisfies \eqref{LeviCivitaConnectionnonflat1ALL}
for the functions $\widetilde a=a/f^2$, $\widetilde c=f^2c$, $\widetilde d=f^2d$ and $\widetilde \beta_2=f\beta_2$. Now, equality \eqref{TheoremnonflatEQUfCPre} gives the second equality in \eqref{TheoremnonflatEQUfC}.

With respect to the new coordinate system $(s,t)$ the integrability conditions \eqref{IntEqAll} become
\begin{subequations}\label{NewIntEqAll}
\begin{eqnarray}
\label{NewIntEq4} \widetilde c_s&=&\widetilde \beta_2,\\
\label{NewIntEq5} \widetilde d_s&=&0,\\
\label{NewIntEq2} \widetilde f \widetilde a_s+\widetilde a_t&=&-\widetilde a\widetilde \beta_2-2\widetilde a\widetilde f_s,\\
\label{NewIntEq6} (\widetilde \beta_2)_s&= & \widetilde a\widetilde d =K=\varkappa,\\
\label{NewIntEq1} \widetilde f_{ss} &=& \widetilde a\widetilde d =K=\varkappa.
\end{eqnarray}
\end{subequations}
Since  $d \neq 0$, from \eqref{NewIntEq5}  we get the second
equality in \eqref{LeviCivitaConnectionnonflatInvariants1ad} for a
non-vanishing smooth function $\lambda_3=\lambda_3(t)$.
Furthermore, combining \eqref{NewIntEq4},  \eqref{NewIntEq6}, and \eqref{NewIntEq1}, we get
\begin{equation}
\label{LeviCivitaConnectionnonflatInvariants1c}  \widetilde c= \widetilde f+\lambda_4s+\lambda_5
\end{equation}
 and
\begin{eqnarray}
\label{NewIntEqAraEq1} \widetilde{\beta}_2- \widetilde f_s=\lambda_4
\end{eqnarray}
 for some smooth functions $\lambda_4=\lambda_4(t)$ and $\lambda_5=\lambda_5(t)$ .

 Next, using  \eqref{QMinimalS421LeviCivita1ALL} we obtain
\begin{equation*}
\begin{array}{lll}
\widetilde{\nabla}_{\bar x} C & = & -\frac{1}{2} \left(2+ \bar x\left(\frac{\widetilde\beta_2}K\right) \right) \bar x \wedge n_1,\\
\widetilde{\nabla}_{\bar y} C & = & -\frac{1}{2} \left( 2\widetilde c+\bar y\left(\frac{\widetilde\beta_2}K\right)+\frac{\widetilde\beta_2}K\left(\widetilde\beta_2- \widetilde f_s\right)  \right) \bar x \wedge n_1.
\end{array}
\end{equation*}
Since $C$ is a constant vector, we have $\widetilde{\nabla}_{\bar x} C =0$ and $\widetilde{\nabla}_{\bar y} C=0$.
So, we get the following two equations:
\begin{subequations}\label{NewCoordPW1Cond1ALL}
\begin{eqnarray}
\label{NewCoordPW1Cond1a} \bar x\left(\frac{\widetilde\beta_2}K\right)&=&-2,\\
\label{NewCoordPW1Cond1b} \bar y\left(\frac{\widetilde\beta_2}K\right)+\frac{\widetilde\beta_2}K\left(\widetilde\beta_2- \widetilde f_s\right) + 2\widetilde c &=&0.
\end{eqnarray}
\end{subequations}
From \eqref{NewCoordPW1Cond1a} using  \eqref{NewIntEq6} we get $\widetilde\beta_2 K_s = 3 K^2$. Differentiating the last equality with respect to $s$, we obtain the equation
$$3 K K_{ss} = 5 K_s^2,$$
whose solution is given by
$$K = \left(-\frac 23 \lambda_1(t) s-\frac 23 \lambda_2(t) \right)^{-3/2}$$
for some smooth functions $\lambda_1 = \lambda_1(t) \neq 0$ and $\lambda_2 = \lambda_2(t)$.
The function $\widetilde\beta_2$ is expressed as $\widetilde\beta_2 = \frac{3}{\lambda_1(t)} K^{1/3}$, which implies
\eqref{LeviCivitaConnectionnonflatInvariants1beta}.
Combining \eqref{NewCoordPW1Cond1Gaussian} and \eqref{NewIntEq1} with the second equality in \eqref{LeviCivitaConnectionnonflatInvariants1ad},
we obtain that $\widetilde a$ is expressed as given by the first equality in \eqref{LeviCivitaConnectionnonflatInvariants1ad}.

Furthermore, from \eqref{LeviCivitaConnectionnonflatInvariants1beta} and \eqref{NewIntEqAraEq1} we get
\begin{equation}
\label{LeviCivitaConnectionnonflatInvariants1f} \widetilde f=-\frac{9}{\lambda_1^2}\left(-\frac 23 \lambda_1(t) s-\frac 23 \lambda_2(t) \right)^{1/2}-\lambda_4(t) s+\lambda_6(t)
\end{equation}
for a smooth function $\lambda_6=\lambda_6(t)$.

Next, using  \eqref{NewCoordPW1Cond1b} and taking into consideration  \eqref{LeviCivitaConnectionnonflatInvariants1beta}, \eqref{LeviCivitaConnectionnonflatInvariants1c}, \eqref{NewIntEqAraEq1} and \eqref{NewCoordPW1Cond1Gaussian},
we get that the function $\lambda_5(t)$ is expressed as follows:
\begin{equation}
\label{LeviCivitaConnectionnonflatInvariantsCond1c} \lambda_5=\left(\frac{\lambda_2}{\lambda_1}\right)'+\frac{\lambda_2\lambda_4}{\lambda_1}.
\end{equation}
Further, using \eqref{LeviCivitaConnectionnonflatInvariants1ad}, \eqref{LeviCivitaConnectionnonflatInvariants1beta},
 \eqref{LeviCivitaConnectionnonflatInvariants1f}, and   \eqref{NewIntEq2}, we obtain
$$s\left(\frac 13 \lambda_1\lambda_3\lambda_4+\lambda_1'\lambda_3-\frac 23 \lambda_1\lambda_3'\right)+\left(\lambda_1\lambda_3\lambda_6+\lambda_2'\lambda_3-\frac 23 \lambda_2\lambda_3'+\frac 43\lambda_2\lambda_3\lambda_4\right) = 0,$$
which implies the following expressions for the functions $\lambda_4(t)$ and $\lambda_6(t)$:
\begin{eqnarray}
\label{LeviCivitaConnectionnonflatInvariantsCond1b}  \lambda_4&=&2\frac{\lambda_3'}{\lambda_3}-3\frac{\lambda_1'}{\lambda_1},\\
\label{LeviCivitaConnectionnonflatInvariantsCond1d} \lambda_6&=&-\frac{\lambda_2'}{\lambda_1} - \frac {2\lambda_2\lambda_3'}{\lambda_1\lambda_3}+ \frac {4\lambda_1'\lambda_2}{\lambda_1^2}.
\end{eqnarray}
Now, from \eqref{LeviCivitaConnectionnonflatInvariants1c}, \eqref{LeviCivitaConnectionnonflatInvariants1f}, \eqref{LeviCivitaConnectionnonflatInvariantsCond1c}, and \eqref{LeviCivitaConnectionnonflatInvariantsCond1b} we obtain \eqref{LeviCivitaConnectionnonflatInvariants1cFinal}. 
Finally, taking into consideration \eqref{LeviCivitaConnectionnonflatInvariants1f}, \eqref{LeviCivitaConnectionnonflatInvariantsCond1b}, and \eqref{LeviCivitaConnectionnonflatInvariantsCond1d} we obtain that the function $\widetilde f$ is expressed as given in \eqref{LeviCivitaConnectionnonflatInvariants1fFinal}.   
Hence, the proof of the necessary condition is completed.

In order to prove the sufficient condition we assume that there exists a coordinate system $(s,t)$ such that \eqref{LeviCivitaConnectionnonflat1ALL} and \eqref{LeviCivitaConnectionnonflatInvariants1ALL} are satisfied. By a straightforward computation using \eqref{E42MargTrapGMLap} one can check that $G$ satisfies \eqref{PW1typeEq1} for the non-constant function
$\phi = -4 \left(-\frac 23 \lambda_1s-\frac 23 \lambda_2\right)^{-3/2}$ and vector $C = -\frac12\left(\bar x \wedge \bar y +n_1\wedge n_2+\frac {\widetilde \beta_2}{K}\bar x\wedge n_1\right)$. A further calculation yields that $C$ is a non-zero constant vector. Hence, $M^2_1$ has proper pointwise 1-type Gauss map of second kind.
\end{proof}

\begin{cor}\label{ACorollaryKDneq0}
Let $M^2_1$ be a quasi-minimal surface with non-flat normal connection in the pseudo-Euclidean space $\mathbb E^4_2$. If the Gauss map of $M^2_1$ satisfies \eqref{PW1typeEq1}, then  $M^2_1$ has proper pointwise 1-type Gauss map of second kind.
\end{cor}

In the following theorem, we obtain a parametrization for the surface described in Proposition \ref{THMNONFLATNORMAL}. This completes  the classification of  quasi-minimal surfaces in $\mathbb E^4_2$ with non-flat normal connection and pointwise 1-type Gauss map.

\begin{thm}\label{THMNONFLATNORMALMainTheo}
Let $M^2_1$ be a quasi-minimal surface in $\mathbb E^4_2$ with non-vanishing normal curvature. Then, $M^2_1$ has pointwise 1-type Gauss map if and only if it is congruent to the surface given by
\begin{align}\label{PW1TYPEPAra}
z(s,t)=- s\lambda_3(t) n_1'(t)-\frac{3 \sqrt{6}  \lambda_3(t)\sqrt{-s \lambda_1(t)}}{\lambda_1^2(t)}n_1(t)+\xi(t).
\end{align}
for some smooth functions $\lambda_1 =\lambda_1(t),\ \lambda_3=\lambda_3(t)$  and some $\mathbb E^4_2$-valued smooth functions $n_1(t)$, $\xi(t)$ satisfying the equations
\begin{eqnarray}\label{MainTheLastEqs}
\langle n_1,n_1\rangle=\langle n_1',n_1'\rangle=\langle n_1,\xi'\rangle=\langle \xi',\xi'\rangle&=0,&\quad
\langle n_1',\xi'\rangle =\frac{1}{\lambda_3},\\
\label{PW1TYPEEqu1} n_1''-\left(\frac{\lambda _3'}{\lambda _3 }-\frac{3 \lambda _1'}{\lambda _1}\right)n_1'+\frac{1}{\lambda _3 }n_1&=0,&
\end{eqnarray}
and
\begin{align}\label{PW1TYPEEqu2}
\begin{split}
 \xi'''+\left(\frac{3 \lambda _3'}{\lambda _3} -\frac{3\lambda _1' }{\lambda _1 }\right)\xi''+ \frac{-3 \lambda _1  \left(\lambda _1'  \lambda _3' +\lambda _3  \lambda _1'' \right)+3 \lambda _3  {\lambda _1'}   ^2+\lambda _1   ^2 \left(2 \lambda _3'' +1\right)}{\lambda _1   ^2\lambda _3} \, \xi'=\zeta,
\end{split}
\end{align}
where $\zeta=\zeta(t)$ is the $\mathbb E^4_2$-valued function given by
\begin{align}\label{PW1TYPEEqu2b}
\begin{split}
\zeta=& 81\frac{8{\lambda _1'}^2 \lambda_3^2 - 2 \lambda _1  \lambda _1'' \lambda_3^2 + \lambda _1^2 \lambda_3 \lambda _3'' - 7\lambda_1 \lambda _1' \lambda_3 \lambda _3' +\lambda _1^2 {\lambda _3'}^2}{\lambda _1^5 \lambda_3} \,n_1+ 162 \frac{\lambda _1  \lambda _3' -2 \lambda _3  \lambda _1'}{\lambda _1^4} \,n_1'.
\end{split}
\end{align}
\end{thm}

\begin{proof}
Let $M^2_1$ be the surface described in Proposition \ref{THMNONFLATNORMAL} for some smooth functions $\lambda_1(t)$ and $\lambda_3(t)$.  Then, the first equation in \eqref{LeviCivitaConnectionnonflat1c} implies $n_1=n_1(t)$. Thus, the first equation in \eqref{LeviCivitaConnectionnonflat1a} and \eqref{LeviCivitaConnectionnonflatInvariants1ad} give the differential equation
$$z_{ss}=\lambda_3\left(-\frac 23 \lambda_1s-\frac 23 \lambda_2\right)^{-3/2}n_1(t)$$
which implies
$$z_s= \frac{3\lambda_3}{\lambda_1} \left(-\frac 23 \lambda_1s-\frac 23 \lambda_2\right)^{-1/2}n_1(t) + \xi _1(t)$$
for some $\mathbb E^4_2$-valued smooth function $\xi_1(t)$. Hence, the position vector $z(s,t)$ is given by
\begin{equation}\label{MainTheoEqu1}
z(s,t)=-\frac{3 \sqrt{6} \lambda _3\sqrt{-s \lambda _1-\lambda _2}}{\lambda_1^2} \,n_1(t) +s \,\xi_1(t)+\xi(t)
\end{equation}
where $\xi(t)$ is a $\mathbb E^4_2$-valued smooth function. Without loss of generality, we may assume that $\lambda _2=0$ by re-defining $s$.

Now, using \eqref{LeviCivitaConnectionnonflat1c}, \eqref{LeviCivitaConnectionnonflatInvariants1ad} and \eqref{LeviCivitaConnectionnonflatInvariants1beta} we obtain
$\xi_1(t) = -\lambda_3 n_1'(t)$. Combining the last equality with \eqref{MainTheoEqu1} we get that $M^2_1$ is congruent to the surface parametrized by  \eqref{PW1TYPEPAra}.

Now, from the first equation in \eqref{LeviCivitaConnectionnonflat1b} we have
$$\frac{\partial}{\partial s}\left(\widetilde f z_s+z_t\right)=n_1.$$
The last equation together with \eqref{LeviCivitaConnectionnonflatInvariants1f} and \eqref{PW1TYPEPAra} imply \eqref{PW1TYPEEqu1}. Note that, since $\bar x=z_s$ and $\bar y=\widetilde f z_s+z_t$ form a pseudo-orthonormal frame field of the tangent bundle of $M^2_1$, \eqref{PW1TYPEPAra} and \eqref{PW1TYPEEqu1} imply \eqref{MainTheLastEqs}.

On the other hand, from the second equality in \eqref{LeviCivitaConnectionnonflat1b} we have
$$n_2=\frac1{\widetilde d}\left(\widetilde\nabla_{\bar y}\bar y-\widetilde f_s{\bar y}-\widetilde c n_1\right).$$
Using \eqref{LeviCivitaConnectionnonflatInvariants1ALL}, \eqref{PW1TYPEEqu1} and $\lambda_2=0$, we compute the right-hand side of this equality and obtain
\begin{align*}\label{LABELNEEDED1}
\begin{split}
n_2=& \frac{1}{2 s \lambda_1^2}\left(2s \lambda_1 (2 \lambda_3' \lambda_1 - 3 \lambda_1' \lambda_3+  3 \sqrt{6} \sqrt{-s \lambda _1} \lambda_3 \right)\xi'+\lambda _3 \xi''+ \frac{\lambda _3}{\lambda_1^3} \left(s \lambda_1^3 - 27 \lambda_3\right)n_1'
\\&
+ \frac{3\lambda _3}{2 s \lambda _1 {}^5}\left( 54 s \lambda_1\left(2 \lambda _3  \lambda _1' - \lambda _1\lambda _3' \right)+ \sqrt{6} \sqrt{-s \lambda _1} \left(s \lambda _1^3-27 \lambda _3 \right)\right)n_1.
\end{split}
\end{align*}
The last equality together with  the second equality in \eqref{LeviCivitaConnectionnonflat1d} and \eqref{PW1TYPEEqu1} imply
\begin{align}\notag
\begin{split}
 \xi'''+ 3\left(\frac{\lambda_3'}{\lambda _3} -\frac{\lambda _1' }{\lambda _1 }\right)\xi''+ \left(\frac{2\lambda_3''}{\lambda_3} - \frac{3\lambda_3'\lambda_1'}{\lambda_1\lambda_3} - \frac{3 \lambda_1''}{\lambda_1} + \frac{3\lambda_1'^2}{\lambda_1^2}+\frac{1}{\lambda_3}  \right)\, \xi' =\\
= 81 \left(\frac{8 \lambda_1'^2 \lambda_3}{\lambda_1^5} - \frac{2\lambda_3  \lambda _1''} {\lambda_1^4} + \frac{\lambda _3''}{\lambda_1^3} - \frac{7 \lambda _1' \lambda _3'}{\lambda_1^4} + \frac{\lambda_3'^2}{\lambda_1^3 \lambda_3} \right)\,n_1+ 162 \left(\frac{\lambda _3'}{\lambda_1^3} - \frac{2 \lambda _3  \lambda _1'}{\lambda _1^4}\right) \,n_1'.
\end{split}
\end{align}
Denoting by $\zeta$ the vector field  given in \eqref{PW1TYPEEqu2b} we obtain \eqref{PW1TYPEEqu2}.
Hence, the proof of the  necessary condition is completed.

The converse follows by a direct computation.
\end{proof}

Below we present an explicit example of a quasi-minimal surface with non-flat normal connection and pointwise 1-type Gauss map.

\begin{Example}
Let $\mathcal{M}$ be the surface given by
\begin{align*}
\begin{split}
z(s,t)=&\left(-4s^{1/2}\cos t+s\sin t+\frac 12\cos t,-4s^{1/2}\sin t-s\cos t+\frac 12\sin t,\right.\\
&-4s^{1/2}\sin t-s\cos t-\frac 12\sin t,\left.-4s^{1/2}\cos t+s\sin t-\frac 12\cos t\right)
\end{split}
\end{align*}
in the pseudo-Euclidean space $\mathbb E^4_2$. Calculating the tangent vector fields $z_s$ and $z_t$ we get $\langle z_s, z_s \rangle = 0$, $\langle z_s, z_t \rangle = -1$, $\langle z_t, z_t \rangle = -8 s^{1/2}$. Hence, $\mathcal{M}$  is a Lorentz surface in $\mathbb E^4_2$.

We consider the following normal vector field $n_1= (\cos t,\sin t,\sin t,\cos t)$.
Note that  $\langle n_1, n_1 \rangle = 0$. Hence, there exists a normal vector field $n_2$ such that $\langle n_2, n_2 \rangle = 0$ and $\langle n_1, n_2 \rangle = -1$.
Now, we consider the following pseudo-orthonormal tangent frame field
$$\bar x =  z_s,\quad \bar y = -4 s^{1/2} z_s + z_t.$$

 By a straightforward  computation
one can see that the mean curvature vector field  is
$H=-n_1$. Hence, $\mathcal{M}$ is a quasi-minimal surface.
Direct computations show that
$$ \widetilde\nabla_{\bar x}n_1=0, \qquad \widetilde\nabla_{\bar y}n_1=-\bar x - 2s^{-1/2} n_1,$$
which imply  $\widetilde \beta_1 = 0$, $\widetilde \beta_2=-2s^{-1/2}$. So, $\mathcal{M}$ is a surface with non-parallel mean curvature vector field. The Gauss curvature $K$ and the normal curvature $\varkappa$ are given by the expressions:
$$K = s^{-3/2}, \quad \varkappa = s^{-3/2}.$$
Hence, $\mathcal{M}$  is a quasi-minimal surface with non-flat normal connection.
 By a straightforward  computation we obtain that
\eqref{PW1typeEq1} is satisfied for the function $\phi = -4 s^{-3/2}$ and the constant vector
$C = -\frac12\left(\bar x \wedge \bar y - 2 s \,\bar x\wedge n_1 + n_1\wedge n_2 \right)$. So, $\mathcal{M}$  is of proper pointwise 1-type Gauss map of second type.

Note that the Levi-Civita connection of $\mathcal{M}$ satisfies \eqref{LeviCivitaConnectionnonflat1ALL}
for the functions $\widetilde a=s^{-3/2}$,  $\widetilde d=1$, $\widetilde \beta_2=-2s^{-1/2}$, $\widetilde c=\widetilde f=-4s^{1/2}$,
which can be obtained by putting $\lambda_1=-\frac 32$, $\lambda_2=0$, $\lambda_3=1$ in  \eqref{LeviCivitaConnectionnonflatInvariants1ALL}.
\end{Example}

\begin{rem}
 We would like to note that in the Minkowski space $\mathbb E^4_1$ all
marginally trapped surfaces with pointwise 1-type Gauss map have flat normal connection, while in the pseudo-Euclidean space $\mathbb E^4_2$ there exist quasi-minimal surfaces with non-flat normal connection and pointwise 1-type Gauss map.
\end{rem}

\begin{rem}
As far as we know, all examples of quasi-minimal surfaces in $\mathbb E^4_2$ known till
now in the literature are surfaces with parallel mean curvature vector field. Here we give an explicit example of a quasi-minimal surface with non-parallel mean curvature
vector field.
\end{rem}

\textbf{Acknowledgments:}
The first author is  partially supported by the National Science Fund,
Ministry of Education and Science of Bulgaria under contract
DFNI-I 02/14. The second  author  is supported by T\"UB\.ITAK  (Project Name: Y\_EUCL2TIP, Project Number: 114F199).

This work was done during the second author's  visit at the Bulgarian Academy of Sciences in June 2015. He would like to express his deepest gratitude to Professors G. Ganchev and V. Milousheva for their warm hospitality during his stay.

\end{document}